\documentclass[11pt,reqno]{amsart}

\usepackage[margin=1in]{geometry}
\usepackage{amsmath}
\usepackage{amssymb}
\usepackage{amsthm}
\usepackage{dsfont}
\usepackage{graphicx} %
\usepackage{microtype}

\usepackage{tikz}
\usetikzlibrary{shapes.geometric, positioning, fit}

\usepackage{charter}
\usepackage[cal=euler]{mathalpha}

\usepackage{xcolor}
\usepackage[backref=page]{hyperref}
\hypersetup{
    colorlinks=true,
    linkcolor=cyan!80!black,
    citecolor=cyan!80!black,
    urlcolor=magenta,
}

\usepackage{booktabs}
\usepackage{makecell}
\usepackage[symbol]{footmisc}

\allowdisplaybreaks

\newcommand{\bbE}{\mathbb{E}}
\newcommand{\bbC}{\mathbb{C}}

\newcommand{\bbP}{\mathbb{P}}

\newcommand{\bbZ}{\mathbb{Z}}
\newcommand{\bbI}{\mathbb{I}}

\newcommand{\cF}{\mathcal{F}}

\newcommand{\cS}{\mathcal{S}}

\newcommand{\Unif}{\mathsf{Uniform}}

\newcommand{\Cay}{\mathrm{Cay}}

\newcommand{\gausshyp}[4]{%
  {}_2F_1 \left[ 
    \begin{array}{c} 
      #1, \; #2 \\ 
      #3 
    \end{array} ;\, #4 
  \right]%
}

\newcommand{\convas}{\xrightarrow{\text{a.s.}}}

\newcommand{\convd}{\xrightarrow{d}}

\let\tilde\widetilde

\theoremstyle{plain}
\newtheorem{theorem}{Theorem}[section]
\newtheorem{corollary}{Corollary}[section]
\newtheorem{proposition}{Proposition}[section]
\newtheorem{lemma}{Lemma}[section]

\theoremstyle{definition}

\theoremstyle{remark}
\newtheorem{remark}{Remark}[section]

\title[Elephant random walk on $\mathbb{Z}_2 * \mathbb{Z}_2$]{Elephant random walk on the infinite dihedral group $\mathbb{Z}_2 * \mathbb{Z}_2$}
\author[S. S. Mukherjee]{Soumendu Sundar Mukherjee}
\author[H. Talukdar]{Himasish Talukdar}
\address{
    Statistics and Mathematics Unit \\
    Indian Statistical Institute \\
    203 B.T. Road, Kolkata 700108 \\
    West Bengal, India.
}
\email{ssmukherjee@isical.ac.in}
\email{talukdar.himasish@gmail.com}

\begin{document}

\begin{abstract}
Elephant random walks were studied recently in \cite{mukherjee2025elephant} on the groups $\mathbb{Z}^{*d_1} * \mathbb{Z}_2^{*d_2}$ whose Cayley graphs are infinite $d$-regular trees with $d = 2d_1 + d_2$. It was found that for $d \ge 3$, the elephant walk is ballistic with the same asymptotic speed $\frac{d - 2}{d}$ as the simple random walk and the memory parameter appears only in the rate of convergence to the limiting speed. In the $d = 2$ case, there are two such groups, both having the bi-infinite path as their Cayley graph. For $(d_1, d_2) = (1, 0)$, the walk is the usual elephant random walk on $\bbZ$, which exhibits anomalous diffusion. In this article, we study the other case, namely $(d_1, d_2) = (0, 2)$, which corresponds to the infinite dihedral group $D_\infty \cong \mathbb{Z}_2 * \mathbb{Z}_2$. Unlike the classical ERW on $\mathbb{Z}$, which is a time-inhomogeneous Markov chain, the ERW on $D_{\infty}$ is non-Markovian. We show that the first and second order behaviours of the \emph{signed location} of the walker agree with those of the simple symmetric random walk on $\bbZ$, with the memory parameter essentially manifesting itself via a lower order correction term that can be written as an explicit functional of the elephant walk on $\mathbb{Z}$. Our result demonstrates that unlike the simple random walk, the elephant walk is sensitive to local algebraic relations. Indeed, although $D_{\infty}$ is virtually abelian, containing $\bbZ$ as a finite-index subgroup, the involutive nature of its generators effectively neutralises memory, thereby ruling out any potential superdiffusive behaviour, in contrast to the superdiffusion observed on its abelian cousin $\bbZ$.

\end{abstract}


\maketitle
\thispagestyle{empty}

\section{Introduction}
The Elephant Random Walk (ERW) was introduced by Sch\"{u}tz and Trimper \cite{schutz2004elephants} as a reinforced random walk on $\bbZ$ with a complete memory of its past. At every step, the walker/elephant chooses to replicate a randomly selected previous step with probability $p$ (the memory parameter), or to move in the opposite direction with probability $1-p$. The ERW on $\bbZ$ is known to exhibit anomalous diffusion \cite{schutz2004elephants, baur2016elephant} and undergoes a phase transition: for $p < 3/4$, the walk is diffusive, while for $p \ge 3/4$, it becomes superdiffusive.

ERWs have also been studied on other infinite graphs, e.g., on $\bbZ^d$ \cite{bercu2019multi, bertenghi2022functional, qin2025recurrence, curien2024recurrence}, on triangular, hexagonal and brickwall lattices \cite{shibata2025functional}, on coverings of dipole graphs \cite{naganuma2026elephant}. Recently, in \cite{mukherjee2025elephant}, the first author generalised the ERW model to finitely generated groups. That work focused on groups whose Cayley graphs are infinite $d$-regular trees with $d \ge 3$, such as the free product $\bbZ^{* d_1} * \bbZ_2^{* d_2}$ (where $d = 2d_1 + d_2$), and it was shown that the walk is ballistic with a speed of $(d - 2)/d$, regardless of the memory parameter $p$. The rate of convergence to this limiting speed appears to depend on $p$, undergoing a phase transition at $p_d = \frac{d + 1}{2d}$.

There are exactly two such groups in the boundary case of $d = 2$: the abelian group $\mathbb{Z}$ (corresponding to $(d_1, d_2) = (1, 0)$) and the infinite dihedral group\footnote[2]{The infinite dihedral group also arises naturally as the semi-direct product $\bbZ \rtimes \bbZ_2$ and serves as the isometry group of $\mathbb{Z}$. Here, however, we shall work with the $\bbZ_2 * \bbZ_2$ model.} $D_\infty \cong \mathbb{Z}_2 * \mathbb{Z}_2$ (corresponding to $(d_1, d_2) = (0, 2)$). In this article, we study the ERW on $D_{\infty}$. While the Cayley graph of $D_\infty$ is the bi-infinite path (an infinite $2$-regular tree), identical to that of $\bbZ$, the different algebraic structure of $D_{\infty}$ leads to a different dynamics. This further motivates a direct comparison with the Simple Symmetric Random Walk (SSRW) on $\bbZ$, which coincides with the ERW at $p = \tfrac{1}{2}$ on both $\bbZ$ and $D_\infty$.

The analysis of the ERW on $\mathbb{Z}_2 * \mathbb{Z}_2$ requires a different approach than the one used in \cite{mukherjee2025elephant}. We utilise the group presentation (two generators $a$ and $b$, both involutions) to describe the process as a certain reinforced random walk on $\bbZ$, which admits a natural coupling with the ERW on $\bbZ$, and can be analysed via martingale techniques to obtain precise limit theorems.

On $d$-regular trees with $d \ge 3$, geometry dictates the first-order behaviour: the exponential growth of the tree pushes the walker outward at the same speed as a simple random walk, regardless of the memory parameter $p$. One might expect that on the bi-infinite path ($d = 2$), the lack of geometric branching could allow the memory parameter to dominate, as seen in the superdiffusive phase of the ERW on $\bbZ$. However, our main result is that the ERW on $D_\infty$ behaves asymptotically like the SSRW on $\bbZ$. Specifically, the effect of the memory parameter $p$ is negligible at first and second orders, manifesting only in a lower-order correction term. Interestingly, this correction term can be described as an explicit functional of the ERW on $\bbZ$.

This highlights a fundamental difference between random walks with memory and ordinary random walks. Random walks on \emph{virtually abelian} infinite groups (such as $D_\infty$) typically mirror the asymptotic behaviour of those on their finite-index abelian subgroups. However, the ERW, which is arguably the simplest and most natural model of a random walk with full memory, is highly sensitive to the underlying group relations. In the abelian case of $\mathbb{Z}$, memory reinforces translational momentum, leading to the superdiffusive regime when it is sufficiently strong ($p \ge 3/4$). On the other hand, since the generators of $\bbZ_2 * \bbZ_2$ are involutions ($a^2 = b^2 =e$), a remembered step often results in an immediate backtrack rather than a continuation of momentum. This ``reflective'' mechanism essentially \emph{stalls} the accumulation of drift, causing the memory effect to cancel out at the first and second orders.

The rest of the paper is organised as follows. In Section~\ref{sec:model}, we describe the model and state our results. Section~\ref{sec:proofs} contains the proofs. Appendix~\ref{sec:aux} contains the proof of an auxiliary complex analytic lemma.

\section{The model and our main results}\label{sec:model}
For a finitely generated group $G$, with a symmetric\footnote[3]{The generating set $\cS$ being symmetric entails that $s \in \cS$ if and only if $s^{-1} \in \cS$.} generating set $\cS$, recall that the (left-) Cayley graph $\Cay(G; \cS)$ of $G$ with respect to $\cS$ is a $|\cS|$-regular graph with vertex set $G$ and an edge between vertices $x, y \in G$ if and only if $yx^{-1} \in \cS$.
Suppose $D_{\infty}$ is the infinite dihedral group presented as 
\[
    D_{\infty} = \langle a,b \mid a^2=e,\, b^2=e\rangle,
\]
where $e$ denotes the identity element. It is easy to see that $\Cay(D_{\infty}; \{a, b\})$ is the bi-infinite path depicted in Figure~\ref{fig:cayley_graph}.
\begin{figure}[!t]
    \centering
    \begin{tikzpicture}[scale=1.2,
    every node/.style={font=\small},
    vertex/.style={circle, draw, minimum size=6mm}
]

\foreach \x/\label in {
    -3/bab,
    -2/ab,
    -1/b,
     0/e,
     1/a,
     2/ba,
     3/aba
}{
    \node[vertex] (v\x) at (\x,0) {};
    \node[above] at (v\x.north) {\(\label\)};
    \node[below] at (v\x.south) {\(\x\)};
}

\foreach \x in {-3,-2,-1,0,1,2}{
    \draw (v\x) -- (v\the\numexpr\x+1\relax);
}

\draw (-4,0) -- (v-3);
\draw (v3) -- (4,0);

\node at (-4.5,0) {\(\cdots\)};
\node at (4.5,0) {\(\cdots\)};

\end{tikzpicture}
    \caption{The (left-) Cayley graph of $D_{\infty}$ with respect to the generators $\{a, b\}$. The integer below a node denote its signed location.}
    \label{fig:cayley_graph}
\end{figure}
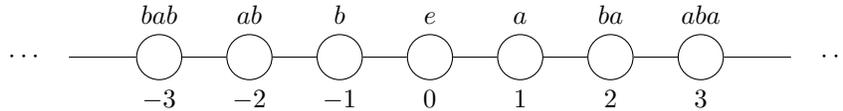

Following \cite{mukherjee2025elephant}, we now define an ERW $(w_n)_{n \ge 0}$ on $D_{\infty}$ (or on $\Cay(D_{\infty}; \{a, b\})$) with steps $(g_n)_{n \ge 0} \in \{a, b\}$. Define $w_0=e$. The first step $g_1$ is chosen uniformly from $\{a, b\}$. Let $\cF_n:= \sigma(w_0, g_1, \ldots, g_n)$ be the natural filtration of $w_0, g_1, \ldots, g_n$. Fix the \emph{memory parameter} $p \in [0,1]$.  Given $\cF_n$, we first pick $I_{n + 1} \sim \Unif(\{1,2, \ldots, n\})$ and define
\begin{equation}\label{eqn:g_n_definition}
    g_{n+1}=
    \begin{cases}
        g_{I_{n + 1}} & \text{with probability } p,\\
        {g}^{(c)}_{I_{n + 1}} & \text{with probability } 1-p,
    \end{cases}
\end{equation} 
where $a^{(c)} := b$ and $b^{(c)} := a$. We denote by $w_n$ the position of the walker at time $n$, which is encoded by the \emph{reduction} of the element $g_n \cdots g_1$ (as per the relations $a^2 = b^2 = e$).

One can alternatively specify the distribution of $(g_n)_{n\ge 1}$ as follows. Let $A_n$ and $B_n$ denote respectively the number of $a$-steps and $b$-steps till time $n$. Given $\cF_n$, we see that
\begin{equation}\label{eqn:g_n_definition_alter}
    g_{n+1} = 
    \begin{cases}
        a &\text{with probability } \frac{p A_n + (1-p) B_n}{n},\\
        b &\text{with probability } \frac{p B_n + (1-p) A_n}{n}.
    \end{cases}
\end{equation}
Let $\rho$ denote the \emph{word metric} (equivalently, the graph distance on the (left-) Cayley graph): for two group elements $x, y \in D_{\infty}$, $\rho(x, y)$ gives the length of the reduced word $yx^{-1}$ (i.e. the number of letters in it). We let $\Delta_{n} := \rho(w_n, e)$ denote the distance of the walker from $e$ at time $n$.

We also define the \emph{signed location} $\Delta_n^{(s)}$ of the walker as follows:
\begin{equation}\label{eq:signed-location}
    \Delta_{n}^{(s)} := \begin{cases}
        \Delta_n & \text{if } \rho(w_n, a) \le \rho(w_n, b), \\
        -\Delta_n & \text{otherwise.}
    \end{cases}
\end{equation}
In other words, we mark the branch of the Cayley tree (rooted at $e$) containing $a$ as positive (see Figure~\ref{fig:cayley_graph}).

In general, $(\Delta_n^{(s)})_{n\ge 0}$ is a \emph{non-Markovian process}. In fact, as we shall see in Remark~\ref{rem:non-Markov} below,
\[
    \bbP(\Delta^{(s)}_{n + 1} = \Delta^{(s)}_{n} \pm 1 \mid \cF_{n}) = \frac{1}{2} \pm (-1)^n \frac{2p - 1}{2n} \left[(-1)^{n - 1} \Delta^{(s)}_n + 2 \sum_{k = 1}^{n - 1} (-1)^{k - 1} \Delta^{(s)}_k\right].
\]
Further, for $p = 1/2$, $(\Delta_n^{(s)})_{n \ge 0}$ has the same distribution as the simple symmetric random walk (SSRW) on $\bbZ$.

\begin{remark}
If $p = 1$, the elephant always repeats its first step. As such, $\Delta_n$ alternates between $0$ and $1$.
\end{remark}

We note here that the process $(W_n)_{n \ge 0}$ given by 
\[
    W_n = A_n - B_n
\]
is the usual ERW on $\bbZ$ (indeed, this is the connection with two colour P\'olya urn models noticed by \cite{baur2016elephant}).

Our main result is the following explicit Doob decomposition of the process $(\Delta_{n}^{(s)})_{n\ge 0}$, which shows that the memory parameter essentially manifests itself via a lower order correction term, which is a functional of $(W_n)_{n \ge 0}$, the ERW on $\bbZ$.
\begin{theorem}\label{thm:main}
There exists a zero-mean martingale $(\Xi_n)_{n \ge 0}$ adapted to $(\cF_{n})_{n \ge 0}$, with bounded increments and its predictable quadratic variation $\langle \Xi \rangle_n$ satisfying $\frac{\langle \Xi \rangle_n}{n} \convas 1$, such that
\begin{equation}\label{eq:doob-decomp}
    \Delta_{n}^{(s)} = \Xi_n + q \sum_{k = 1}^{n - 1} (-1)^k \frac{W_k}{k},
\end{equation}
where $q = 2p - 1 \in [-1, 1)$. The predictable process $(Z_n)_{n \ge 0}$, where $Z_n := \Delta_{n}^{(s)} - \Xi_n$, converges almost surely and in $L^2$ to the zero-mean random variable
\[
    Z_{\infty} := q \sum_{k = 1}^{\infty} (-1)^k \frac{W_k}{k},
\]
with variance
\begin{equation}\label{eq:var_Z_infty}
    \bbE[Z_{\infty}^2] = \begin{cases} \frac{q^2}{2q-1}
    \int_0^1\frac{u^{-2q + 1} - 1}{1 - u} \cdot \frac{(1 - \frac{u}{2})^{q - 1} - 1}{u} \, du & \text{if } q \ne \frac{1}{2}, \\[5pt]
    \frac{1}{4} \int_0^1 \frac{-\ln u}{1-u}\cdot\frac{\left(1-\frac{u}{2}\right)^{-1/2}-1}{u}\,du & \text{if } q = \frac{1}{2}.
    \end{cases}
\end{equation}
\end{theorem}
In Figure~\ref{fig:var-Z-inf}, we plot $\bbE[Z_{\infty}^2]$ as a function of $q$, showing its dependence on the memory parameter.

\begin{figure}
    \centering
    \includegraphics[width=0.6\linewidth]{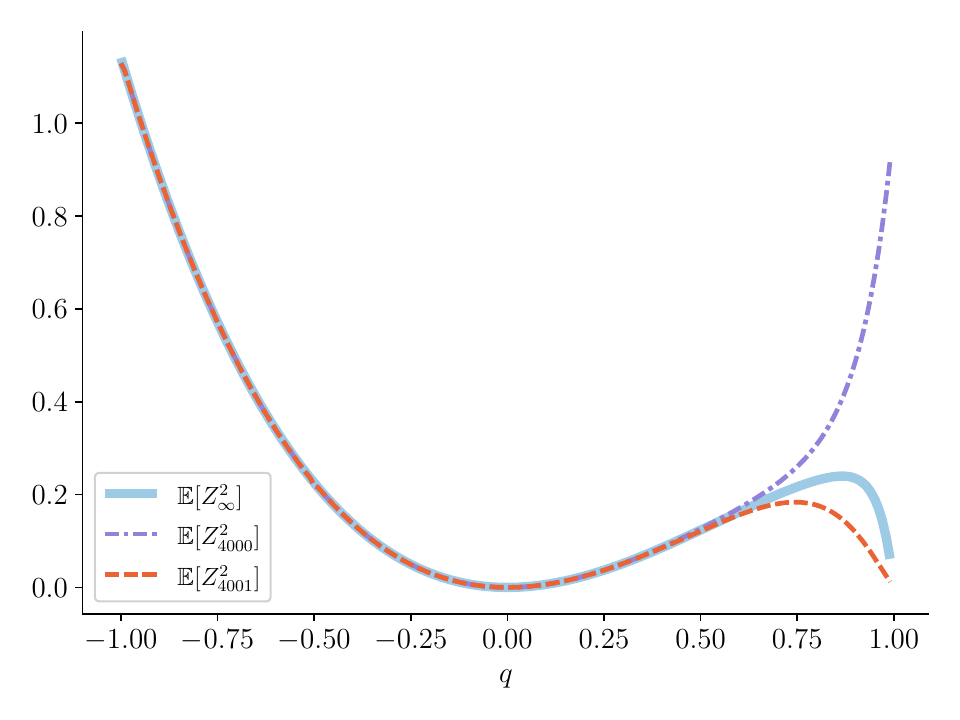}
    \caption{$\bbE[Z_{\infty}^2]$ as a function of $q = 2p - 1$. The plot illustrates the slow rate of convergence for $q$ close to $1$.}
    \label{fig:var-Z-inf}
\end{figure}

From Theorem~\ref{thm:main}, we see that the first and second order behaviours of the ERW on $\bbZ_2 * \bbZ_2$ match those of the SSRW on $\bbZ$. 

\begin{corollary}\label{cor:S_n}
    For any $p \in [0, 1)$, the following hold.
    \begin{enumerate}
        \setlength{\itemsep}{5pt}
        \item[(a)] \textbf{Strong Law:} $\frac{\Delta_n^{(s)}}{n} \convas 0$.
        \item[(b)] \textbf{Functional Central Limit Theorem:} 
        The process $\big(\frac{1}{\sqrt{n}}\Delta_{\lfloor nt \rfloor}^{(s)}\big)_{t \in [0, 1]}$ converges weakly to the standard Brownian motion $(B_t)_{t \in [0, 1]}$ in the Skorohod space of c\`adl\`ag functions on $[0, 1]$. In particular, $\frac{\Delta_n^{(s)}}{\sqrt{n}} \convd N(0, 1)$. %
        \item[(c)] \textbf{Law of the Iterated Logarithm:} We have 
        \[
            \limsup_{n \to \infty} \frac{\Delta_n^{(s)}}{\sqrt{2n \log{\log{n}}}} = - \liminf_{n \to \infty} \frac{\Delta_n^{(s)}}{\sqrt{2n \log{\log{n}}}} = 1.
        \]
        As a consequence, $\Delta_n^{(s)}$ is recurrent. 
        
        \item[(d)] \textbf{Quadratic Strong Law:} $\frac{1}{\log{n}} \sum_{k=1}^n \frac{(\Delta_k^{(s)})^2}{k^2} \convas 1$.
    \end{enumerate}
\end{corollary}

\section{Proofs}\label{sec:proofs}
We first give an alternative description of the signed location process $(\Delta_n^{(s)})$. To that end, we define a reinforced walk $(S_n)_{n \ge 0}$ on $\bbZ$, with $S_0:=0$ and increments $X_n = S_n - S_{n - 1}, n \ge 1$, as follows. For a realisation $(g_n)_{n \ge 1}$ of the ERW, let
\begin{equation}\label{eqn:X_n_definition}
    X_{n+1} := 
    \begin{cases}
        1 & \text{if $n+1$ is odd and $g_{n+1} = a$, or $n+1$ is even and $g_{n+1}=b$},\\
        -1 & \text{if $n+1$ is odd and $g_{n+1} = b$, or $n+1$ is even and $g_{n+1}=a$}.
    \end{cases}
\end{equation}
In other words, at odd epochs we encode the $a$ steps  as $1$ and the $b$ steps as $-1$, and flip the signs at the even epochs. The key observation is that $S_n$ is an alternative description of $\Delta_n^{(s)}$.
\begin{proposition}\label{prop:Delta^{(s)}_n=S_n}
    For all $n \ge 0$, $\Delta^{(s)}_n = S_n$.
\end{proposition}
\begin{proof}
    Define $\delta^{(s)}_n := \Delta^{(s)}_n - \Delta^{(s)}_{n-1}$ for $n \ge 1$. We prove, by strong induction on $n$, that $\delta^{(s)}_n = X_n$ for all $n \ge 1$. Clearly, $\delta^{(s)}_1 = X_1$. Now suppose that $\delta^{(s)}_k = X_k$ for $k \le n$ which automatically implies that $\Delta^{(s)}_k = S_k$ for $k \le n$. We shall show that $\delta^{(s)}_{n+1} = X_{n+1}$. We will have three cases depending on the sign of $\Delta^{(s)}_n$.
    
    \textbf{Case 1:} Suppose $\Delta^{(s)}_n = S_n =0$. Then 
    \[
        \delta^{(s)}_{n+1} = 
        \begin{cases}
        1 & \text{if } g_{n+1} = a, \\
        -1 & \text{if } g_{n+1} = b. 
        \end{cases}
    \]
    Since $S_n = 0$, $n + 1$ must be odd. So, from \eqref{eqn:X_n_definition}, we see that $\delta^{(s)}_{n+1} = X_{n+1}$.
    
    \textbf{Case 2:} Suppose $\Delta^{(s)}_n = S_n>0$. Define $\tau_n$ to be the last time before $n$, when $w_k$ was $e$, i.e. $\tau_n := \max\{0 \le k \le n-1 : w_k =e\}$. Then, of course, $\Delta^{(s)}_{\tau_n} = S_{\tau_n} = 0$. 
    
    Note that in our dynamics, at each epoch the newly arrived letter either gets cancelled by the leftmost letter of the existing string, or gets appended to it. Thus for each letter of the current string, there is an associated epoch when it first appeared. We call $\zeta_n$ to be this epoch for the leftmost letter of $w_n$; that is, we define $\zeta_n$ to be the first epoch after $\tau$, when the string $w_n$ appeared and its leftmost letter did not get cancelled during the time period from $\zeta_n$ to $n$. Formally,
    \[
        \zeta_n := \min\{\tau_n +1 \le k \le n: w_k = w_n,   \Delta^{(s)}_{\ell} \ge \Delta^{(s)}_n \,\, \text{for} \,\, k+1 \le \ell \le n\}.
    \]
    For example, if the process progresses as follows:
    \begin{align*}
    w_0 &= e, \\
    w_1 &= a, \\
    w_2 &= ba, \\
    w_3 &= aba, \\
    w_4 &= baba, \\
    w_5 &= aba, \\
    w_6 &= baba, \\
    w_7 &= ababa, \\
    w_8 &= baba, \\
    w_9 &= aba, \\
    &\,\,\,\vdots
    \end{align*}
    then
    \[
        \zeta_7 = 7, \quad \zeta_8 = 6, \quad \zeta_9 = 5.
    \]
    We note that
    \begin{equation}\label{eq:del_n_plus_1}
        \delta^{(s)}_{n+1} = \bbI(g_{n+1} \ne g_{\zeta_n}) - \bbI(g_{n+1} = g_{\zeta_n}).
    \end{equation}
    Observe that $X_{\zeta_n} = \delta^{(s)}_{\zeta_n}$ by induction hypothesis, which in turn must be $1$, because otherwise, the minimality in the definition of $\zeta_n$ will be violated. So knowing the parity of $\zeta_n$ determines $g_{\zeta_n}$. Further note that since $w_{\zeta_n} = w_n$ and cancellations always occur in pairs, $n - \zeta_n$ must be even. Hence, $n + 1 - \zeta_n$ is odd. Thus
    \begin{align}\nonumber
        X_{n+1} &= \bbI(n+1\,\text{is odd},\, g_{n+1}=a) + \bbI(n+1\,\text{is even},\, g_{n+1}=b) \\ \nonumber
        &\qquad -\bbI(n+1\,\text{is odd},\, g_{n+1}=b) -\bbI(n+1\,\text{is even},\, g_{n+1}=a) \\ \nonumber
        &= \bbI(\zeta_n \,\text{is even},\, g_{n+1}=a) + \bbI(\zeta_n \,\text{is odd}, \, g_{n+1}=b) \\ \nonumber
        &\qquad -\bbI(\zeta_n \,\text{is even},\, g_{n+1}=b) - \bbI(\zeta_n\,\text{is odd}, \, g_{n+1}=a) \\ \nonumber
        &= \bbI(g_{\zeta_n} = b,\, g_{n+1}=a) + \bbI(g_{\zeta_n} = a,\, g_{n+1}=b) \\ \nonumber
        &\qquad -\bbI(g_{\eta_n} = b,\, g_{n+1}=b) -\bbI(g_{\zeta_n} = a,\, g_{n+1}=a) \\ \label{eq:X_n_plus_1}
        &= \bbI(g_{n+1} \ne g_{\eta}) - \bbI(g_{n+1} = g_{\eta}),
    \end{align}
    where the third equality follows from the fact that $X_{\zeta_n} = 1$. If follows from \eqref{eq:del_n_plus_1} and \eqref{eq:X_n_plus_1} that $\delta^{(s)}_{n+1} = X_{n+1}$.

    \textbf{Case 3:} $\Delta^{(s)}_n = S_n<0$. This can be treated in a similar manner as Case 2.
    
    The proof is therefore complete by strong induction.
\end{proof}

In view of Proposition~\ref{prop:Delta^{(s)}_n=S_n}, we shall consider $(S_n)_{n \ge 0}$ as the signed location process. In fact, \eqref{eqn:X_n_definition} gives an immediate coupling between the ERW on $D_{\infty}$ and the ERW on $\bbZ$. Let $W_n = A_n - B_n = 2 A_n - n$ denote the ERW on $\bbZ$. Both $A_n$ and $W_n$ are time-inhomogeneous Markov chains.
Let $\Delta$ denote the difference operator: For any real sequence $(x_n)_{n \ge 0}$, $\Delta x_n = x_n - x_{n - 1}$ for $n \ge 1$. Then, given $\cF_n$
\begin{equation}\label{eq:W-dynamics}
    \Delta W_{n + 1} = \begin{cases}
        1 & \text{with probability } \frac{1}{2} + \frac{q W_n}{2n}, \\ 
        -1 & \text{with probability } \frac{1}{2} - \frac{q W_n}{2n}.
    \end{cases}
\end{equation}
From \eqref{eqn:X_n_definition}, we note that
\begin{equation}\label{eq:rel-S-W-1st-diff}
    \Delta S_n = \begin{cases}
        \Delta W_n & \text{if $n$ is odd}, \\
        -\Delta W_n & \text{if $n$ is even}. 
    \end{cases}
\end{equation}
These relations intuitively explain the lack of superdiffusion in $(S_n)_{n \ge 0}$. The drift $\Delta W_{n + 1}$ essentially cancels $\Delta W_{n}$.

\begin{remark}\label{rem:non-Markov}
    We may also observe that $S_n$, unlike $W_n$, is \emph{not} a time-inhomogeneous Markov chain. In fact, it is highly non-Markovian in nature. Indeed, rewriting \eqref{eq:rel-S-W-1st-diff} as
    \[
        \Delta W_n = (-1)^{n - 1} \Delta S_n,
    \]
    we get
    \[
        W_n = \sum_{k = 1}^n (-1)^{k - 1} \Delta S_k = (-1)^{n - 1} S_n + 2 \sum_{k = 1}^{n - 1} (-1)^{k - 1} S_k.
    \]
    But then,
    \[
        \bbP(S_{n + 1} = S_{n} \pm 1 \mid \cF_{n}) = \bbP(\Delta W_{n + 1} = \pm (-1)^n \mid \cF_{n}) = \frac{1}{2} \pm (-1)^n \frac{q W_{n}}{2n}
    \]
    depends on the full path $(S_1, \ldots, S_{2n})$ in a non-trivial manner.
\end{remark}
    
We now collect some properties of $W_n$ from \cite{bercu2017martingale}, which will be useful in what follows. First of all, we note that $\bbE[W_n] = 0$ for any $n \ge 0$.
\begin{proposition}\label{prop:W_n}
    We have the following.
    \begin{enumerate}
    \setlength{\itemsep}{5pt}
        \item [(a)] \cite[Theorem~3.2]{bercu2017martingale} For $p \in [0, 3/4)$, we have
        \[
            \limsup_{n \to \infty} \frac{|W_n|}{\sqrt{2n\log \log n}} = \frac{1}{\sqrt{3 - 4p}}, a.s.
        \]
        \item [(b)] \cite[Theorem~3.5]{bercu2017martingale} For $p = 3/4$, we have
        \[
            \limsup_{n \to \infty} \frac{|W_n|}{\sqrt{2n\log n \log \log \log n}} = 1, a.s.
        \]
        \item [(c)] \cite[Theorem~3.7]{bercu2017martingale} For $p \in (3/4, 1)$, we have
        \[
            n^{-q} |W_n| \convas L,
        \]
        for some non-degenerate random variable $L$.
    \end{enumerate}
\end{proposition}
Let
\[
    r_n := \begin{cases}
        n^{1/2} & \text{if } 0 \le p < 3/4, \\
        (\frac{n}{\log n})^{1/2} & \text{if } p = 3/4, \\
        n^{2(1 - p)} & \text{if } 3/4 < p < 1.
    \end{cases}
\]
The following result follows from explicit formula for the variance of the ERW on $\bbZ$ (see, e.g., Eq. (14)-(15) in \cite{schutz2004elephants}, or Eq. (B.19) and the unnumbered equation preceding it in \cite{bercu2017martingale}).
\begin{proposition}\label{prop:W_n_var}
We have %
\begin{equation}\label{eq:W_n_var}
    \bbE [W_k^2] = \frac{\Gamma(k + 2q)}{ \Gamma(k)} \sum_{l = 1}^{k} \frac{\Gamma(l)}{\Gamma(l + 2q)} \\
    = H(k, q),
\end{equation}
where %
\begin{equation}\label{eq:H_defn}
    H(k, q) := \begin{cases}
            \frac{k}{(2q - 1)}\big(\frac{\Gamma(k + 2q)}{\Gamma(k + 1) \Gamma(2q)}-1\big) & \text{if } q \ne 1/2, \\[0.5em]
            k \sum_{j = 1}^{k} \frac{1}{j} & \text{if } q = 1/2.
        \end{cases}
\end{equation}
Consequently, for any $p \in [0, 1)$, there is a constant $C_p > 0$, such that for all $n \ge 1$,
\begin{equation}\label{eq:W_n_var_bd}
    \bbE W_n^2 \le C_p n^2 r_n^{-2}.
\end{equation}
\end{proposition}
In these formulas appearing in \eqref{eq:W_n_var} and \eqref{eq:H_defn}, we adopt the convention $\frac{1}{\Gamma(0)} = \frac{1}{\Gamma(-1)} = 0$.

We are now ready to prove Theorem~\ref{thm:main}.
\begin{proof}[Proof of Theorem~\ref{thm:main}]
Using \eqref{eq:W-dynamics} and \eqref{eq:rel-S-W-1st-diff}, we compute
\[
    \bbE[\Delta S_{n + 1} \mid \cF_n] = (-1)^n \frac{q W_n}{n}.
\]
It follows that 
\begin{equation}\label{eq:xi_defn}
    \xi_n := \Delta S_n - (-1)^{n - 1} \frac{q W_{n - 1}}{n - 1}
\end{equation}
is a martingale difference sequence, with $|\xi_n| \le 2$. We define
\[
    \Xi_n := \begin{cases}
        \sum_{k = 1}^n \xi_k & \text{if } n \ge 1, \\
        0 & \text{if } n = 0,
    \end{cases}
\]
to be the corresponding martingale.
Then we may write
\begin{equation}\label{eq:S_n-representation}
    S_{n} = \Xi_n + q \sum_{k = 1}^{n - 1} (-1)^k\frac{W_k}{k},
\end{equation}
which is the Doob deomposition of $S_n$.
We now compute the predictable quadratic variation $\langle \Xi \rangle_n = \sum_{k = 1}^n \bbE[\xi_k^2 \mid \cF_{k - 1}]$. Notice that
\begin{align*}
    \bbE[\xi_{n + 1}^2 \mid \cF_{n}] &= \bbE[ (\Delta S_{n + 1})^2 \mid \cF_{n}] - \frac{q^2 W_n^2}{n^2} \\
    &= 1 - \frac{q^2 W_n^2}{n^2}.
\end{align*}
Therefore
\begin{align}\label{eq:qv-decomp} \nonumber
    \frac{\langle \Xi \rangle_n}{n}
    &= \frac{1}{n}\sum_{k=1}^{n} \bigg(1 - \frac{q^2 W_{k - 1}^2}{(k - 1)^2}\bigg)\\
    &= 1 - \frac{q^2}{n}\sum_{k=1}^{n - 1} \frac{W_k^2}{k^2}.
\end{align}
Now by Proposition~\ref{prop:W_n}, for any $p \in [0, 1)$, we have $\frac{W_n^2}{n^2}=o(1)$ almost surely. A fortiori, the Ces\`{a}ro average $\frac{1}{n}\sum_{k=1}^{n-1}\frac{W_k^2}{k^2}$ is also $o(1)$ almost surely, and hence $\frac{\langle \Xi \rangle_n}{n} \convas 1$.

Define
\begin{equation}\label{eq:Z_n-tilde-defn}
    \tilde{Z}_n := \sum_{k = 1}^{n - 1} (-1)^k\frac{W_k}{k},
\end{equation}
so that
\begin{equation}\label{eq:S_n-representation-short}
    S_{n} = \Xi_n + q \tilde{Z}_n.
\end{equation}
We shall show that $\tilde{Z}_n$ converges a.s. and in $L^2$ to some random variable $\tilde{Z}_{\infty}$.
Noting that
\begin{equation}\label{eq:W_k_conditional}
    \bbE[W_{k} \mid \cF_{k-1}] = W_{k - 1} + q \frac{W_{k - 1}}{k - 1},
\end{equation}
let
\[
    \delta_k = W_k - \left(1 + \frac{q}{k - 1}\right) W_{k - 1}
\]
denote the corresponding martingale difference sequence. We have $|\delta_k| \le 1 + |q| \le 2$. Consider
\[
    \tilde{Z}_{2n + 1} = \sum_{k = 1}^n \bigg(\frac{W_{2k}}{2k} - \frac{W_{2k} - 1}{2k - 1}\bigg).
\]
Notice that
\begin{align*}
    \frac{W_{2k}}{2k} - \frac{W_{2k} - 1}{2k - 1} &= \frac{\delta_{2k} + W_{2k - 1}(1 + \frac{q}{2k - 1})}{2k} - \frac{W_{2k - 1}}{2k - 1} \\
    &= \frac{\delta_{2k}}{2k} + \bigg(\frac{1}{2k} + \frac{q}{2k (2k - 1)} - \frac{1}{2k - 1}\bigg) W_{2k - 1} \\
    &= \frac{\delta_{2k}}{2k} - \frac{1 - q}{2k (2k - 1)} W_{k - 1}.
\end{align*}
Thus
\begin{equation}\label{eq:Z_tilda}
    \tilde{Z}_{2n + 1} = \sum_{k = 1}^n \frac{\delta_{2k}}{2k} - (1 - q)\sum_{k = 1}^n \frac{W_{2k - 1}}{2k(2k - 1)}.
\end{equation}
By virtue of Proposition~\ref{prop:W_n}, for any $p\in[0,1)$, there exists $\epsilon>0$ such that $\limsup_{k \to \infty} \frac{|W_k|}{k^{1-\epsilon}} < \infty$. It follows that
\begin{equation}\label{eq:W-series-abs-conv}
    \sum_{k = 1}^\infty \frac{|W_{2k - 1}|}{2k(2k - 1)} < \infty.
\end{equation}
Thus the second term in \eqref{eq:Z_tilda} converges almost surely. Moreover,
\begin{align*}
    \bbE&\left[\left(\sum_{k = 1}^\infty \frac{|W_{2k - 1}|}{2k(2k - 1)}\right)^2\right] \\
    &\le \sum_{k=1}^\infty \frac{\bbE |W_{2k-1}^2|}{(2k-1)^4} + \sum_{k \ne l} \frac{\bbE |W_{2k-1} W_{2l-1}|}{(2k-1)^2 (2l-1)^2} \\
    &\le \sum_{k=1}^\infty \frac{\bbE W_{2k-1}^2}{(2k-1)^4} + \sum_{k \ne l} \frac{(\bbE W_{2k-1}^2)^{1/2} (\bbE W_{2l-1}^2)^{1/2}}{(2k-1)^2 (2l-1)^2} \quad \text{(by Cauchy-Schwarz)} \\
    &\le C_p\sum_{k=1}^\infty \frac{1}{(2k-1)^2 r_{2k-1}^2} + C_p\sum_{k \ne l} \frac{1}{(2k-1) (2l-1) r_{2k-1} r_{2l-1}} \quad \text{(using~\eqref{eq:W_n_var_bd})} \\
    &= C_p\bigg(\sum_{k=1}^\infty \frac{1}{(2k-1) r_{2k-1}}\bigg)^2 \\
    &< \infty.
\end{align*}
It follows by dominated convergence that the second term in \eqref{eq:Z_tilda} also converges in $L^2$ to its a.s. limit. 

We now show that the first term in \eqref{eq:Z_tilda} converges a.s. and in $L^2$. To that end, notice that $M_n = \sum_{k = 1}^n \frac{\delta_{2k}}{2k}$ is in fact a martingale w.r.t. the filtration $(\cF_{2n})_{n \ge 0}$. Indeed,
\[
    \bbE[M_n \mid \cF_{2n - 2}] = M_{n - 1} + \frac{1}{2n}\bbE[\delta_{2n} \mid \cF_{2n - 2}] = M_{n - 1}
\]
since by the tower property,
\[
    \bbE[\delta_{2n} \mid \cF_{2n - 2}] = \bbE[ \bbE[\delta_{2n} \mid \cF_{2n - 1}] \mid \cF_{2n - 2}] = 0.
\]
Moreover, $M_n$ is $L^2$-bounded; in fact, for any $n \ge 1$,
\begin{equation}\label{eq:M-L2-bdd}
    \bbE[M_n^2] \le \sum_{k \ge 1} \bbE[(\Delta M_k)^2] = \sum_{k \ge 1}\bbE[\delta_{2k}^2 / (2k)^2] \le \sum_{k \ge 1} \frac{1}{k^2} = \frac{\pi^2}{6}.
\end{equation}
It follows that $M_n$ converges a.s. and in $L^2$ to some random variable $M_{\infty}$.

Thus $\tilde{Z}_{2n + 1}$ converges a.s. and in $L^2$ to the random variable 
\[
    \tilde{Z}_{\infty} = M_{\infty} - (1 - q) \sum_{k = 1}^\infty \frac{W_{2k - 1}}{2k(2k - 1)}.
\]
Since
\[
    \tilde{Z}_{2n + 2} = \tilde{Z}_{2n + 1} - \frac{W_{2n + 1}}{2n + 1},
\]
and $\frac{W_n}{n}$ converges a.s. and in $L^2$ to $0$ (see Propositions~\ref{prop:W_n} and ~\ref{prop:W_n_var}), we conclude that $\tilde{Z}_{2n + 2}$ also converges a.s. and in $L^2$ to the same limit $\tilde{Z}_{\infty}$. Therefore $\tilde{Z}_n$ converges a.s. and in $L^2$ to $\tilde{Z}_{\infty}$. 
It is clear that
\[
    \bbE[\tilde{Z}_{\infty}] = \lim_{n \to \infty} \bbE[\tilde{Z}_n] = 0.
\]
The claimed variance formula \eqref{eq:var_Z_infty} follows from the equivalent integral representation for $\bbE[\tilde{Z}_{\infty}^2]$ obtained in Proposition~\ref{prop:var-computation}, which also shows that $\tilde{Z}_{\infty}$ is non-degenerate.
\end{proof}

\begin{proof}[Proof of Corollary~\ref{cor:S_n}]
Recall the decomposition \eqref{eq:S_n-representation-short}, and the facts that $(\Xi_n)_{n \ge 0}$ is a martingale with bounded increments, $\frac{\langle \Xi \rangle_n}{n} \convas 1$, and $\tilde{Z}_n \convas \tilde{Z}_{\infty}$. It is clear that for parts (a) and (c), $\tilde{Z}_n$ has negligible contributions. Hence, part (a) follows immediately from the strong law of large numbers for martingales (see, e.g., Theorem~1.3.15 in \cite{duflo1997random}). Part (c) also follows directly from Stout's law of the iterated logarithm for martingales \cite[Theorem~3]{stout1970martingale}.

Now we prove part (b). By the functional central limit theorem for martingales \cite[Theorem 18.2]{billingsley2013convergence}, we may conclude that the process $(\frac{\Xi_{\lfloor nt \rfloor}}{\sqrt{n}})_{t\in [0,1]}$ converges weakly to the standard Brownian motion $(B_t)_{t \in [0, 1]}$ in the space $D[0, 1]$ of c\`adl\`ag functions on $[0, 1]$ equipped with the Skorohod $J_1$ topology. From \eqref{eq:S_n-representation-short}, we see that
\[
    \sup_{t \in [0, 1]}\bigg|\frac{\Xi_{\lfloor nt \rfloor}}{\sqrt{n}} - \frac{S_{\lfloor nt \rfloor}}{\sqrt{n}}\bigg| \le \frac{1}{\sqrt{n}}\max_{1 \le k \le n} |\tilde{Z}_k|.
\]
We shall show that $\max_{1 \le k \le n} |\tilde{Z}_k| = O_P(1)$. First note from \eqref{eq:Z_tilda} that for any odd integer $2m + 1 \le n$,
\begin{align}\label{eq:Z_odd_unif_bd} \nonumber
    |\tilde{Z}_{2m + 1}| &\le |M_m| + (1 - q) \sum_{k = 1}^{m} \frac{|W_{2k - 1}|}{2k (2k - 1)} \\
    &\le \max_{1 \le m \le n} |M_m| + (1 - q) \sum_{k = 1}^{\infty} \frac{|W_{2k - 1}|}{2k (2k - 1)}.
\end{align}
As shown earlier in \eqref{eq:W-series-abs-conv}, the last series is convergent. By Doob's maximal inequality and \eqref{eq:M-L2-bdd},
\[
    \bbP\left(\max_{1 \le m \le n} |M_m| \ge \lambda\right) \le \frac{\bbE [M_n^2]}{\lambda^2} \le \frac{\pi^2}{6\lambda^2},
\]
It therefore follows from \eqref{eq:Z_odd_unif_bd} that
\[
    \max_{\substack{1 \le k \le n \\ k \text{ odd}}} |\tilde{Z}_k| = O_P(1).
\]
Writing
\[
    \tilde{Z}_{2m + 2} = - W_1 - \sum_{k = 1}^m \left(\frac{W_{2k + 1}}{2k + 1} - \frac{W_{2k}}{2k}\right) = - W_1 - \tilde{M}_m + (1 - q) \sum_{k = 1}^m \frac{W_{2k}}{(2k + 1) 2k},
\]
where $\tilde{M}_m := \sum_{k = 1}^m \frac{\delta_{2k + 1}}{2k + 1}$ is a $(\cF_{2m + 1})_{m \ge 1}$-martingale, a
similar analysis yields that
\[
    \max_{\substack{1 \le k \le n \\ k \text{ even}}} |\tilde{Z}_k| = O_P(1).
\]
We conclude that
\[
    \max_{1 \le k \le n} |\tilde{Z}_k| = O_P(1).
\]
It follows that the process $\big(\frac{S_{\lfloor nt \rfloor}}{\sqrt{n}}\big)_{t\in [0,1]}$ converges weakly to $(B_t)_{t \in [0, 1]}$.

Finally, we prove (d). Using the representation \eqref{eq:S_n-representation-short}, we have
\begin{align*}
    \frac{1}{\log{n}} \sum_{k=1}^n \frac{S_{k}^2}{k^2} &= \frac{1}{\log{n}} \sum_{k=1}^n \frac{(\Xi_k + q \tilde{Z}_k)^2}{k^2}\\
    &= \frac{1}{\log n} \sum_{k=1}^n \frac{\Xi_k^2}{k^2} + \frac{q^2}{\log n}\sum_{k=1}^n \frac{\tilde{Z}_k^2}{k^2} + \frac{2q}{\log n} \sum_{k=1}^n \frac{\Xi_k \tilde{Z}_k}{k^2}.
\end{align*}
From this we see using Cauchy-Schwarz that
\begin{align}\label{eq:QSL_proof_perturbation}
    \bigg|\frac{1}{\log n} \sum_{k=1}^n \frac{S_{k}^2}{k^2} -\frac{1}{\log n} \sum_{k=1}^n \frac{\Xi_k^2}{k^2}\bigg| \le \frac{1}{\log n}\sum_{k=1}^n \frac{\tilde{Z}_k^2}{k^2} + 2 \bigg(\frac{1}{\log n}\sum_{k=1}^n \frac{\Xi_k^2}{k^2}\bigg)^{1/2}\bigg(\frac{1}{\log n}\sum_{k=1}^n \frac{\tilde{Z}_k^2}{k^2}\bigg)^{1/2}.
\end{align}
By the quadratic strong law for martingales,
\[
    \frac{1}{\log n} \sum_{k=1}^n \frac{\Xi_k^2}{k^2} \convas 1.
\]
(This version is a direct consequence of Theorem~3 of \cite{chaabane2001invariance} and the law of the iterated logarithm.)
Moreover, since $\tilde{Z}_k \convas \tilde{Z}_{\infty}$, we have 
\[
    \frac{1}{\log n}\sum_{k=1}^n \frac{\tilde{Z}_k^2}{k^2} \convas 0.
\]
It follows that the right-hand side of  \eqref{eq:QSL_proof_perturbation} vanishes a.s. in the large-$n$ limit.
\end{proof}
Finally, we derive the advertised explicit integral representation for $\bbE[\tilde{Z}_{\infty}^2]$.
\begin{proposition}\label{prop:var-computation}
    We have
    \begin{equation}\label{eq:Z-tilde-var-repr}
        \bbE[\tilde{Z}_{\infty}^2] = \begin{cases}
        \frac{1}{2q-1}\int_0^1\frac{u^{-2q + 1} - 1}{1 - u} \cdot \frac{(1 - \frac{u}{2})^{q - 1} - 1}{u} du & \text{if } q \ne \frac{1}{2}, \\[5pt]
        \int_0^1 \frac{-\ln u}{1-u}\cdot\frac{\left(1-\frac{u}{2}\right)^{-1/2}-1}{u}\,du & \text{if } q = \frac{1}{2}.
        \end{cases}
    \end{equation}
\end{proposition}

The proof of Proposition~\ref{prop:var-computation} is non-trivial, and makes crucial use of Euler's integral formula for Gauss's hypergeometric function ${}_2F_1$. Let $(x)_n := x (x + 1) \cdots (x + n - 1) = \Gamma(x+n)/\Gamma(x)$ denote the Pochhammer symbol for rising factorials. We recall the definition of Gauss's hypergeometric function:
\begin{equation*}
    \gausshyp{a}{b}{c}{z} := \sum_{n=0}^\infty \frac{(a)_n (b)_n}{(c)_n} \cdot \frac{z^n}{n!},
\end{equation*}
where $c$ is not zero or a negative integer. Euler's integral representation (see, e.g., Theorem~2.2.1 of \cite{andrews1999special}) says that
\begin{equation}\label{eq:euler-representation}
    \gausshyp{a}{b}{c}{z} = \frac{\Gamma(c)}{\Gamma(b)\Gamma(c - b)} \int_0^1 t^{b - 1} (1 - t)^{c - b -1} (1 - zt)^{-a} \, dt,
\end{equation}
which is valid for $c > b > 0$.

The proof of Proposition~\ref{prop:var-computation} will be based on four lemmas. The first one gives a decomposition of the variance of $\tilde{Z}_n$.
\begin{lemma}\label{lem:T1-T2}
Define
\begin{align}\label{eq:T1-defn}
    T_{1, n, q} &:= \sum_{k = 1}^n \frac{H(k,q) I(k, q)}{k^2} \cdot \int_0^1 \frac{u^{k+q-1} (1-u)^{1 - q}}{1 + u} \, du, \\ \label{eq:T2-defn}
    T_{2, n, q} &:= \int_0^1 \frac{u^{n + q} (1 - u)^{-q}}{1 + u}\, du \sum_{k = 1}^n (-1)^{n - k} \frac{H(k, q) I(k, q)}{k^2},
\end{align}
where $H(k, q)$ is as in \eqref{eq:H_defn} and
\begin{equation}\label{eq:I-defn}
    I(k, q) := \bigg(\int_0^1 u^{k + q - 1} (1 - u)^{-q} \, du\bigg)^{-1} = \frac{\Gamma(k + 1)}{\Gamma(k + q)\Gamma(1 - q)}.
\end{equation}
Then for any $n \ge 0$,
\begin{equation}\label{eq:breakdown-into-T1-T2}
    \bbE[\tilde{Z}_{n + 1}^2] = T_{1, n, q} + 2 \, T_{2, n, q}.
\end{equation}
\end{lemma}

It turns out that for any $q \in [-1, 1)$, $T_{2, n, q}$ vanishes in the large-$n$ limit. In fact, we can establish the following quantitative bounds (notice the slow rate of convergence for $q$ close to $1$; this is reflected in Figure~\ref{fig:var-Z-inf}).
\begin{lemma}\label{lem:T2-limit}
We have
\begin{equation}\label{eq:T2_rates}
    |T_{2, n, q}|
    =
    \begin{cases}
        O(n^{-1})        & \text{if } q \in [-1,\,0), \\
        O(n^{-1}\log n)  & \text{if } q = 0, \\
        O(n^{q-1})       & \text{if } q \in (0,\,1).
    \end{cases}
\end{equation}
\end{lemma}
Using Gauss's hypergeometric function (in particular, Euler's integral representation), we can evaluate the limit of $T_{1, n, q}$ for $q \in (0, 1)$.
\begin{lemma}\label{lem:T1-limit}
We have
\begin{equation}\label{eq:T_1_limit}
    \lim_{n \to \infty} T_{1, n, q} = \begin{cases}
    \frac{1}{2q-1}\int_0^1\frac{u^{-2q + 1} - 1}{1 - u} \cdot \frac{(1 - \frac{u}{2})^{q - 1} - 1}{u} du & \text{if } q \in (0, 1) \setminus \{\frac{1}{2}\}, \\[5pt]
    \int_0^1 \frac{-\ln u}{1-u}\cdot\frac{\left(1-\frac{u}{2}\right)^{-1/2}-1}{u}\,du & \text{if } q = \frac{1}{2}.
    \end{cases}
\end{equation}
\end{lemma}
Finally, we shall use an analytic continuation argument to tackle the regime $q \in [-1, 0]$. In the following lemma, $q$ will be treated as a complex number. 
\begin{lemma}\label{lem:T1-unif-conv}
Let $\eta \in (0, \frac{1}{2})$ and define the rectangular domain
\[
    R_\eta := \left\{z \in \bbC : -1 - \eta < \Re(z) < \frac{1}{2} - \eta, |\Im(z)| < 1 \right\}.
\]
Then $T_{1, n, q}$, defined for $q \in R_{\eta}$ by the same formula \eqref{eq:T1-defn}\footnote{Thus both $H(k, q)$ and $I(k, q)$ are also being defined by the same formula for complex $q$}, converges uniformly as a function of $q$.
\end{lemma}
We are ready to prove Proposition~\ref{prop:var-computation}.
\begin{proof}[Proof of Proposition~\ref{prop:var-computation}]
Since $\tilde{Z}_n$ converges in $L^2$ to $\tilde{Z}_{\infty}$, using Lemmas~\ref{lem:T1-T2} and \ref{lem:T2-limit}, we see that for any $q \in [-1, 1)$,
\begin{align*}
    \bbE[\tilde{Z}_{\infty}^2] = \lim_{n \to \infty} \bbE[\tilde{Z}_{n + 1}^2] = \lim_{n \to \infty} [T_{1, n, q} + 2 \, T_{2, n, q}] = \lim_{n \to \infty} T_{1, n, q}.
\end{align*}
Then Lemma~\ref{lem:T1-limit} implies the desired representation \eqref{eq:Z-tilde-var-repr} for $q \in (0, 1)$.
However, by its very definition, for each $n$, $T_{1, n, q}$ is an analytic function of $q$ in the domain $R_{\eta}$. Therefore, as a consequence of Lemma~\ref{lem:T1-unif-conv}, it converges to some analytic function on $R_{\eta}$. However,  for $q \in (0, \frac{1}{2} - \eta)$, the pointwise limit of $T_{1, n, q}$ is the integral $\frac{1}{2q-1}\int_0^1\frac{u^{-2q + 1} - 1}{1 - u} \cdot \frac{(1 - \frac{u}{2})^{q - 1} - 1}{u} du$, which itself is an analytic function of $q \in R_\eta$ (see Lemma~\ref{lem:phi_holomorphic} for a proof). We conclude that $\lim_{n \to \infty} T_{1, n, q}$ agrees with this integral for all $q \in R_{\eta}$, and in particular for any $q \in [-1, 0]$. This completes the proof.
\end{proof}

\begin{proof}[Proof of Lemma~\ref{lem:T1-T2}]
Recalling the definition of $\tilde{Z}_n$ in \eqref{eq:Z_n-tilde-defn}, we have
\begin{align*}
    \bbE[\tilde{Z}_{n + 1}^2] = \sum_{k, l = 1}^n \frac{(-1)^{k + l}}{k l} \bbE[W_k W_l].
\end{align*}
From \eqref{eq:W_k_conditional} it follows that
\[
    \bbE[W_{i+1}| \cF_{i}] = \left(1 + \frac{q}{i}\right) W_i,
\]
which, for $k < l$, gives
\[
    \bbE[W_l|\cF_{k}] = \prod_{i=k}^{l-1} \left(1 + \frac{q}{i}\right) W_k = \frac{(k + q)_{l - k}}{(k)_{l - k}} W_k.
\]
Hence for $k <l$, we get using \eqref{eq:W_n_var} that 
\begin{align*}
    \bbE[W_k W_l] &= \bbE[W_k \bbE[W_l|\cF_k]]\\
    &= \frac{(k + q)_{l - k}}{(k)_{l - k}} \bbE[W_k^2]\\
    &= \frac{(k + q)_{l - k}}{(k)_{l - k}} H(k, q).
\end{align*}
Using this, we write
\begin{align}\label{eq:var-Ztilde} \nonumber
    \bbE[\tilde{Z}_{n + 1}^2] &= \sum_{k, l = 1}^n \frac{(-1)^{k + l}}{k l} \bbE[W_k W_l] \\ \nonumber
    &= \sum_{k=1}^n \frac{1}{k^2} \bbE[W_k^2] + 2\sum_{k = 1}^n \sum_{l=k+1}^n \frac{(-1)^{k + l}}{k l} \bbE[W_k W_l] \\ \nonumber
    &= \sum_{k=1}^n \frac{1}{k^2} \bbE[W_k^2] + 2\sum_{k = 1}^n \sum_{l=1}^{n-k} \frac{(-1)^{l}}{k(k+l)} \bbE[W_k W_{k+l}] \\ \nonumber
    &= \sum_{k=1}^n \frac{1}{k^2} H(k,q) + 2\sum_{k = 1}^n \sum_{l=1}^{n-k} \frac{(-1)^{l}}{k(k+l)} \frac{(k + q)_l}{(k)_l} H(k,q) \\ 
    &= \sum_{k = 1}^n \frac{H(k,q)}{k^2} \left(1 +  2 \sum_{l=1}^{n-k} (-1)^{l} \frac{(k + q)_{l}}{(k + 1)_{l}} \right).
\end{align}
We notice that the Pochhammer ratio $\frac{(k + q)_l}{(k + 1)_l}$ may be written as a beta-integral:
\[
    \frac{(k+q)_l}{(k+1)_l} =  I(k, q) \cdot \int_0^1 u^{l+k+q-1} (1-u)^{-q} \, du,
\]
where $I(k, q)$ is as in \eqref{eq:I-defn}. This representation is valid for $l = 0$, as long as we interpret the left-hand side as $1$.
Using these observations, we write
\begin{align*}
    &1 + 2\sum_{l = 1}^{n - k} (-1)^{l} \frac{(k + q)_{l}}{(k + 1)_{l}} \\
    &= 2\sum_{l = 0}^{n - k} (-1)^{l} \frac{(k + q)_{l}}{(k + 1)_{l}} - 1 \\
    &= 2 I(k, q) \cdot \sum_{l = 0}^{n - k}  (-1)^l \int_0^1 u^{l+k+q-1} (1-u)^{-q} \, du - 1 \\
    &= 2 I(k, q) \cdot \int_0^1 u^{k+q-1} (1-u)^{-q} \sum_{l = 0}^{n - k}  (-u)^l \, du - 1\\
    &= 2 I(k, q) \cdot \int_0^1 u^{k+q-1} (1-u)^{-q} \frac{1 - (-u)^{n - k + 1}}{1 + u}\, du - I(k, q)  \cdot \int_0^1 u^{k+q-1} (1-u)^{-q} \, du \\
    &= 2 I(k, q) \cdot \int_0^1 u^{k+q-1} (1-u)^{-q} \bigg( \frac{1}{1 + u} - \frac{1}{2}\bigg) \, du + 2 (-1)^{n - k} I(k, q) \cdot \int_0^1 \frac{u^{n + q} (1 - u)^{-q}}{1 + u}\, du \\
    &= I(k, q) \cdot \int_0^1 \frac{u^{k+q-1} (1-u)^{1 - q}}{1 + u} \, du + 2 (-1)^{n - k} I(k, q) \cdot \int_0^1 \frac{u^{n + q} (1 - u)^{-q}}{1 + u}\, du.
\end{align*}
Plugging this into \eqref{eq:var-Ztilde}, we obtain the desired breakdown \eqref{eq:breakdown-into-T1-T2}.
\end{proof}

\begin{proof}[Proof of Lemma~\ref{lem:T2-limit}]
Observe that
\begin{equation}\label{eq:est-T2-integral}
    \int_0^1 \frac{u^{n + q} (1 - u)^{-q}}{1 + u}\, du \le \int_0^1 u^{n + q} (1 - u)^{-q}\, du = \frac{\Gamma(n + q + 1) \Gamma(1 - q)}{\Gamma(n + 2)} = O(n^{-(1 - q)}).
\end{equation}
Recalling the growth rate of $H(k, q)$, we have that
\begin{equation}\label{eq:a_k}
    a_k := \frac{H(k, q) I(k, q)}{k^2} \asymp \begin{cases}
        k^{-q} & \text{if } q \in [-1, 1/2), \\
        k^{-q} \log k & \text{if } q = 1/2, \\
        k^{-(1 - q)} & \text{if } q \in (1/2, 1).
    \end{cases}
\end{equation}
Let $D_j = \sum_{i = 1}^j (-1)^i$. Then by summation by parts
\begin{align}\label{eq:summation-by-parts} \nonumber
    |\sum_{k = 1}^n (-1)^k a_k| &= |D_n a_n - \sum_{k = 1}^{n - 1} D_k (a_{k + 1} - a_k)| \\
    &\le a_n + \sum_{k = 1}^{n - 1} |a_{k + 1} - a_k|.
\end{align}
We shall now obtain explicit upper bounds on $|a_{k+1}-a_k|$.

First suppose that $q \ne \frac{1}{2}$. Writing
\[
    a_k = C_{1, q} \cdot \frac{\Gamma(k+2q)}{k\,\Gamma(k+q)}
    - C_{2, q} \cdot \frac{\Gamma(k)}{\Gamma(k+q)}
    =: C_{1, q} \cdot b_k - C_{2, q} \cdot c_k,
\]
with $C_{1, q} = \frac{1}{(2q-1)\Gamma(2q)\Gamma(1-q)}$ and $C_{2, q} = \frac{1}{(2q-1)\Gamma(1-q)}$,
we estimate the differences of $b_k$ and $c_k$ separately.

Notice first that
\begin{align*}
    b_{k+1} - b_k &= \frac{\Gamma(k+1+2q)}{(k+1)\Gamma(k+1+q)} - \frac{\Gamma(k+2q)}{k\,\Gamma(k+q)} \\
    &= \frac{\Gamma(k+2q)}{\Gamma(k+q)}\left(\frac{k+2q}{(k+1)(k+q)} - \frac{1}{k}\right) \\
    &= \frac{\Gamma(k+2q)}{\Gamma(k+q)}\cdot \frac{k(k+2q) - (k+1)(k+q)}{k(k+1)(k+q)} \\
    &= \frac{\Gamma(k+2q)}{\Gamma(k+q)}\cdot\frac{k(q-1) - q}{k(k+1)(k+q)}.
\end{align*}
Thus
\begin{equation}\label{eq:diff-bk}
    |b_{k + 1} - b_k| = O(k^{q-2}).
\end{equation}
On the other hand,
\[
    c_{k+1} - c_k = \frac{\Gamma(k+1)}{\Gamma(k+1+q)} - \frac{\Gamma(k)}{\Gamma(k+q)}
    = \frac{\Gamma(k)}{\Gamma(k+q)}\left(\frac{k}{k+q} - 1\right)
    = -\frac{q\,\Gamma(k)}{(k+q)\,\Gamma(k+q)}.
\]
Thus
\begin{equation}\label{eq:diff-ck}
    |c_{k + 1} - c_k| = O(k^{-q - 1}).
\end{equation}
Together, \eqref{eq:diff-bk} and \eqref{eq:diff-ck} give
\begin{equation}\label{eq:a_k-diff}
    |a_{k+1}-a_k| = O(k^{q-2}) + O(k^{-q-1}).
\end{equation}

When $q = \frac{1}{2}$, we have
\[
    a_k = \frac{\Gamma(k)}{\Gamma(k + \frac{1}{2}) \Gamma(\frac{1}{2})} \sum_{j = 1}^{k} \frac{1}{j},
\]
so that
\begin{align}\label{eq:a_k-diff-1/2}\nonumber
    |a_{k + 1} - a_{k}| &= \bigg|\frac{\Gamma(k + 1)}{\Gamma(k + \frac{3}{2}) \Gamma(\frac{1}{2})} \sum_{j = 1}^{k + 1} \frac{1}{j} -\frac{\Gamma(k)}{\Gamma(k + \frac{1}{2}) \Gamma(\frac{1}{2})} \sum_{j = 1}^{k} \frac{1}{j}\bigg| \\ \nonumber
    &=\frac{\Gamma(k + 1)}{\Gamma(k + \frac{3}{2}) \Gamma(\frac{1}{2})} \frac{1}{k + 1} + \frac{\Gamma(k)}{\Gamma(k + \frac{1}{2}) \Gamma(\frac{1}{2})} \bigg(1 - \frac{k}{k + \frac{1}{2}} \bigg)\sum_{j = 1}^{k} \frac{1}{j} \\
    &=O(k^{-\frac{3}{2}} \log k).
\end{align}

Summing the estimates \eqref{eq:a_k-diff} or \eqref{eq:a_k-diff-1/2} over $k$, we see that
\[
    \sum_{k=1}^{n-1}|a_{k+1}-a_k| =
    \begin{cases}
        O(n^{-q})  & \text{if } q \in [-1,\,0), \\
        O(\log n)  & \text{if } q  = 0, \\
        O(1)       & \text{if } q \in (0,\,1).
    \end{cases}
\]
Plugging these estimates and the ones in \eqref{eq:a_k} into \eqref{eq:summation-by-parts}, we obtain
\begin{equation}\label{eq:a_k-alternating-sum}
    \left|\sum_{k=1}^n(-1)^k a_k\right| =
    \begin{cases}
        O(n^{-q})  & \text{if } q \in [-1,\, 0), \\
        O(\log n)  & \text{if } q = 0, \\
        O(1)       & \text{if } q \in (0, 1).
    \end{cases}
\end{equation}
Combining \eqref{eq:est-T2-integral} and \eqref{eq:a_k-alternating-sum}, we obtain the stated estimates \eqref{eq:T2_rates} on $|T_{2, n, q}|$.
\end{proof}

\begin{proof}[Proof of Lemma~\ref{lem:T1-limit}]
By monotone convergence,
\[
    \lim_{n \to \infty} T_{1, n, q} = \sum_{k = 1}^{\infty} \frac{H(k,q) I(k, q)}{k^2} \cdot  \int_0^1 \frac{u^{k+q-1} (1-u)^{1 - q}}{1 + u} \, du.
\]
Since the quantity inside the sum and the integral is non-negative, we can interchange the sum and the integral to get
\begin{equation}\label{eq:T1-limit-expression-1}
    \lim_{n \to \infty} T_{1, n, q} = \frac{1}{\Gamma(1 - q)}\int_0^1  \frac{(1 - u)^{1 - q}}{1 + u} \sum_{k = 1}^\infty \frac{H(k, q)}{k^2} \cdot \frac{\Gamma(k + 1)}{\Gamma(k + q)} \cdot u^{k + q - 1} \, du.
\end{equation}
Let us now simplify the inner sum in \eqref{eq:T1-limit-expression-1}. For $q \ne 1/2$, from \eqref{eq:H_defn}, we get 
\begin{align*}
    &\sum_{k = 1}^\infty \frac{H(k, q)}{k^2} \cdot \frac{\Gamma(k + 1)}{\Gamma(k + q)} \cdot u^{k + q - 1}\\
    &= \frac{u^{q-1}}{2q - 1} \sum_{k = 1}^\infty \left(\frac{\Gamma(k + 2q)}{\Gamma(k + 1) \Gamma(2q)} - 1\right) \cdot \frac{\Gamma(k + 1)}{\Gamma(k + q)} \cdot \frac{u^{k}}{k}\\
    &= \frac{u^{q-1}}{2q - 1} \sum_{k = 1}^\infty \left(\frac{\Gamma(k + 2q)}{\Gamma(k + 1) \Gamma(2q)} - 1\right) \cdot \frac{\Gamma(k + 1)}{\Gamma(k + q)} \cdot \int_{0}^u x^{k - 1} \, dx\\
    &= \frac{u^{q-1}}{2q - 1} \int_0^u \bigg(\sum_{k = 1}^\infty \frac{\Gamma(k + 2q)}{\Gamma(k + q) \Gamma(2q)} x^{k - 1} - \sum_{k = 1}^\infty \frac{\Gamma(k + 1)}{\Gamma(k + q)} x^{k - 1} \bigg) \, dx.
\end{align*}
In the rest of the proof we shall crucially use the hypothesis $q > 0$. Notice that
\begin{align*}
    \sum_{k = 1}^\infty \frac{\Gamma(k + 2q)}{\Gamma(k + q) \Gamma(2q)} x^{k - 1} &= \sum_{k = 0}^\infty \frac{\Gamma(k + 1 + 2q)}{\Gamma(k + 1 + q) \Gamma(2q)} x^{k} \\
    &= \frac{2q}{\Gamma(q + 1)} \, \gausshyp{2q + 1}{1}{q + 1}{x} \\
    &= \frac{2q}{\Gamma(q + 1)} \cdot \frac{\Gamma(q+1)}{\Gamma(q)}\int_0^1 (1-t)^{q-1} (1-xt)^{-2q-1} \, dt \\
    &= \frac{2q}{\Gamma(q)} \int_0^1 (1-t)^{q-1} (1-xt)^{-2q-1} \, dt,
\end{align*}
where in the penultimate line we have used Euler's integral representation \eqref{eq:euler-representation} (notice that here we need $q > 0$).
Similarly,
\begin{align*}
    \sum_{k = 1}^\infty \frac{\Gamma(k + 1)}{\Gamma(k + q)} x^{k - 1} &= \frac{1}{\Gamma(q + 1)} \, \gausshyp{2}{1}{q + 1}{x} \\
    &=\frac{1}{\Gamma(q + 1)} \cdot \frac{\Gamma(q+1)}{\Gamma(q)} \int_0^1 (1-t)^{q-1} (1-xt)^{-2} \, dt \\
    &= \frac{1}{\Gamma(q)} \int_0^1 (1-t)^{q-1} (1-xt)^{-2} \, dt.
\end{align*}
Hence
\begin{align*}
    &\sum_{k = 1}^\infty \frac{H(k, q)}{k^2} \cdot \frac{\Gamma(k + 1)}{\Gamma(k + q)} \cdot u^{k + q - 1} \\
    &= \frac{u^{q-1}}{(2q - 1) \Gamma(q)} \int_0^u \int_0^1 \bigg[2q (1-t)^{q-1} (1-xt)^{-2q-1} - (1-t)^{q-1} (1-xt)^{-2} \bigg] \, dt \, dx \\
    &= \frac{u^{q-1}}{(2q - 1) \Gamma(q)} \int_0^1 (1-t)^{q-1} \int_0^u \bigg[2q (1-xt)^{-2q-1} - (1-xt)^{-2} \bigg] \, dx \, dt \\
    &= \frac{u^{q-1}}{(2q - 1) \Gamma(q)} \int_0^1 \frac{(1-t)^{q-1}}{t}\left[(1-ut)^{-2q}-(1-ut)^{-1}\right] \, dt.
\end{align*}
In the above computation, the interchange in the third line is justified by dominated convergence as $q > 0$. Thus, we have shown that
\begin{align*}
    &\lim_{n \to \infty} T_{1, n, q} \\
    &= \frac{1}{\Gamma(1 - q)}\int_0^1  \frac{(1 - u)^{1 - q}}{1 + u} \sum_{k = 1}^\infty \frac{H(k, q)}{k^2} \cdot \frac{\Gamma(k + 1)}{\Gamma(k + q)} \cdot u^{k + q - 1} \, du \\
    &= \frac{1}{\Gamma(1 - q)}\int_0^1  \frac{(1 - u)^{1 - q}}{1 + u} \cdot \frac{u^{q-1}}{(2q - 1) \Gamma(q)} \int_0^1 \frac{(1-t)^{q-1}}{t}\left[(1-ut)^{-2q}-(1-ut)^{-1}\right] \, dt \, du \\
    &= \frac{1}{(2q - 1) \Gamma(q)\Gamma(1 - q)}\int_0^1  \frac{u^{q-1}(1 - u)^{1 - q}}{1 + u} \int_0^1 \frac{(1-t)^{q-1}}{t}\left[(1-ut)^{-2q}-(1-ut)^{-1}\right] \, dt \, du.
\end{align*}
With the substitution $v = ut$, the inner integral can be rewritten as follows:
\[
    \int_0^1 \frac{(1-t)^{q-1}}{t}\left[(1-ut)^{-2q}-(1-ut)^{-1}\right] \, dt = u^{1-q}\int_0^u \frac{(u-v)^{q-1}}{v}
    \left[(1-v)^{-2q}-(1-v)^{-1}\right]dv.
\]
Therefore
\begin{align}\label{eq:lim-T1-pre-final} \nonumber
    &\lim_{n \to \infty} T_{1, n, q} \\ \nonumber
    &= \frac{1}{(2q-1)\Gamma(q)\Gamma(1-q)}
    \int_0^1 \frac{(1-u)^{1-q}}{1+u}
    \int_0^u \frac{(u-v)^{q-1}}{v}
    \left[(1-v)^{-2q}-(1-v)^{-1}\right]dv\,du \\
    &= 
    \frac{1}{(2q-1)\Gamma(q)\Gamma(1-q)} \int_0^1 \frac{(1-v)^{-2q}-(1-v)^{-1}}{v} \int_v^1 \frac{(1-u)^{1-q}(u-v)^{q-1}}{1+u}\,du\,dv,
\end{align}
where the second equality follows from a 
swapping of the order of integration. Finally, the substitution
$u = v + (1 - v) s$ helps us evaluate the inner integral. Indeed, letting $\lambda = \frac{1 - v}{1 + v}$, we see that
\begin{align*}
    \int_v^1 \frac{(1-u)^{1-q}(u-v)^{q-1}}{1+u}\,du &= (1 -v)\int_0^1\frac{s^{q-1}(1-s)^{1-q}}{1 + v + (1-v)s}\,ds \\
    &= \lambda \int_0^1\frac{s^{q-1}(1-s)^{1-q}}{1 + \lambda s}\,ds \\ 
    &= \lambda \cdot \Gamma(q)\Gamma(2-q) \cdot \gausshyp{1}{q}{2}{-\lambda} \\
    &= \Gamma(q)\Gamma(1-q) \cdot (1 - q) \lambda \cdot \gausshyp{q}{1}{2}{-\lambda} \\
    &= \Gamma(q)\Gamma(1-q) \cdot (1 - q)\lambda \cdot \int_0^1 (1 + \lambda s)^{-q} \, ds\\
    &= \Gamma(q)\Gamma(1-q) \cdot \left[(1 + \lambda)^{1 - q} - 1\right]\\
    &= \Gamma(q)\Gamma(1-q) \left[\left(\frac{2}{1+v}\right)^{1-q}-1\right],
\end{align*}
where the second and the fourth equalities follow from Euler's integral representation \eqref{eq:euler-representation}. Substituting this into \eqref{eq:lim-T1-pre-final}, we get
\begin{align*}
    \lim_{n \to \infty} T_{1, n, q} = \frac{1}{2q-1}\int_0^1\frac{(1-v)^{-2q}-(1-v)^{-1}}{v} \left[\left(\frac{2}{1+v}\right)^{1-q}-1\right] \, dv.
\end{align*}
Finally, the change of variable $u = 1 - v$ yields
\[
    \lim_{n \to \infty} T_{1, n, q} = \frac{1}{2q-1}\int_0^1\frac{u^{-2q + 1} - 1}{1 - u} \cdot \frac{(1 - \frac{u}{2})^{q - 1} - 1}{u} \, du.
\]
Finally, we derive the limit at $q = \frac{1}{2}$. For $k \ge 1$, let $h_k := \sum_{i=1}^k \frac{1}{i}$ denote the $k$-th harmonic number. Note that for $|x| <1$, the generating function of the sequence $(h_k)_{k \ge 1}$ is given by
\begin{equation}\label{eq:harmonic_generating_function}
    \sum_{k=1}^\infty h_k x^k = - \frac{\ln (1-x)}{1-x}.
\end{equation}
Indeed, for $|x| < 1$,
\begin{align*}
    \sum_{k=1}^\infty h_k x^k &= \sum_{k=1}^\infty \sum_{i=1}^k \frac{1}{i} x^k = \sum_{i=1}^\infty \frac{1}{i} \sum_{k=i}^\infty x^k = \frac{1}{1-x} \sum_{i=1}^\infty \frac{x^i}{i} = - \frac{\ln (1-x)}{1-x}.
\end{align*}
Then,
\begin{align*}
    &\lim_{n \to \infty} T_{1, n, \frac{1}{2}} \\
    &= \frac{1}{\Gamma(\frac{1}{2})}\int_0^1  \frac{(1 - u)^{\frac{1}{2}}}{1 + u} \sum_{k = 1}^\infty \frac{H(k, \frac{1}{2})}{k^2} \cdot \frac{\Gamma(k + 1)}{\Gamma(k + \frac{1}{2})} \cdot u^{k + \frac{1}{2} - 1} \, du \\
    &= \frac{1}{\Gamma(\frac{1}{2})}\int_0^1  \frac{(1 - u)^{\frac{1}{2}}}{1 + u} \sum_{k = 1}^\infty \frac{\Gamma(k)}{\Gamma(k + \frac{1}{2})}h_k \cdot u^{k - \frac{1}{2}} \, du\\
    &= \frac{1}{\Gamma^2(\frac{1}{2})}\int_0^1  \frac{(1 - u)^{\frac{1}{2}}}{1 + u} \sum_{k = 1}^\infty \left(\int_0^1 t^{k-1} (1-t)^{-\frac{1}{2}} \, dt\right) h_k \cdot u^{k - \frac{1}{2}} \, du\\
    &= \frac{1}{\Gamma^2(\frac{1}{2})}\int_0^1 \int_0^1 \frac{(1 - u)^{\frac{1}{2}}}{(1 + u) u^{\frac{1}{2}} (1- t)^{\frac{1}{2}} t} \sum_{k = 1}^\infty h_k (ut)^k \, dt \, du \\
    &= \frac{1}{\Gamma^2(\frac{1}{2})}\int_0^1 \int_0^1 \frac{(1 - u)^{\frac{1}{2}}}{(1 + u) u^{\frac{1}{2}} (1-t)^{\frac{1}{2}} t} \left( - \frac{\ln(1-ut)}{1-ut} \right) \, dt \, du \quad (\text{using}\,\, \eqref{eq:harmonic_generating_function}) \\
    &= - \frac{1}{\Gamma^2(\frac{1}{2})}\int_0^1 \int_0^u \frac{(1 - u)^{\frac{1}{2}} \ln(1-v)}{(1 + u) (u - v)^{\frac{1}{2}} v(1-v)} \, dv \, du  \quad (\text{substituting}\,\, v = ut) \\ 
    &= - \frac{1}{\Gamma^2(\frac{1}{2})}\int_0^1 \frac{\ln(1-v)}{v(1-v)} \int_v^1 \frac{(1 - u)^{\frac{1}{2}}}{(1 + u) (u - v)^{\frac{1}{2}}} \, du \, dv \\
    &= - \frac{1}{\Gamma^2(\frac{1}{2})}\int_0^1 \frac{\ln(1-v)}{v(1 + v)} \int_0^1 \frac{(1 - s)^{\frac{1}{2}}}{(1 + \lambda s) s^{\frac{1}{2}}} \, ds \, dw \quad (\text{substituting}\,\, s=\frac{u-v}{1-v} \,\, \text{and letting}\,\, \lambda:= \frac{1-v}{1+v})\\
    &= - \frac{1}{\Gamma^2(\frac{1}{2})}\int_0^1 \frac{\ln(1-v)}{v(1 + v)} \cdot \frac{\Gamma(\frac{1}{2}) \Gamma(\frac{3}{2})}{\Gamma(2)} \, \gausshyp{1}{\frac{1}{2}}{2}{-\lambda} \, dv \quad (\text{using Euler's integral representation}) \\
    &= - \frac{1}{2}\int_0^1 \frac{\ln(1-v)}{v(1 + v)} \, \gausshyp{\frac{1}{2}}{1}{2}{-\lambda} \, dv \\
    &= - \frac{1}{2} \int_0^1 \frac{\ln(1-v)}{v(1 + v)} \int_0^1 \frac{1}{(1+\lambda s)^{\frac{1}{2}}} \, ds \, dv  \quad (\text{using Euler's integral representation}) \\
    &= - \int_0^1 \frac{\ln(1-v)}{v(1 + v)} \cdot \frac{1}{\lambda} \left((1 + \lambda)^{\frac{1}{2}} - 1\right) \, dw \\
    &= - \int_0^1 \frac{\ln(1-v)}{v(1-v)} \left( \sqrt{\frac{2}{1+v}} -1\right) \, dv \\
    &= \int_0^1 \frac{ - \ln{u}}{(1-u)} \cdot \frac{\left(1-\frac{u}{2}\right)^{-\frac{1}{2}}-1}{u} \, du \quad (\text{substituting}\,\, u = 1 - v).
\end{align*}
This completes the proof of \eqref{eq:T_1_limit} for $q=\frac{1}{2}$, finishing the proof of Lemma~\ref{lem:T1-limit}.
\end{proof}
\begin{remark} 
The following heuristic computation also produces the same answer for $q = \frac{1}{2}$; however, we are unable to justify the change of limit and integral. We have from \eqref{eq:T1-limit-expression-1},
\begin{align*}
    \lim_{n \to \infty} T_{1, n, \frac{1}{2}} &= \frac{1}{\Gamma(1/2)}\int_0^1  \frac{(1 - u)^{1/2}}{1 + u} \sum_{k = 1}^\infty \frac{H(k, 1/2)}{k^2} \cdot \frac{\Gamma(k + 1)}{\Gamma(k + 1/2)} \cdot u^{k + 1/2 - 1} \, du \\
    &= \lim_{q \to \frac{1}{2}-} \frac{1}{\Gamma(1 - q)}\int_0^1  \frac{(1 - u)^{1 - q}}{1 + u} \sum_{k = 1}^\infty \frac{H(k, q)}{k^2} \cdot \frac{\Gamma(k + 1)}{\Gamma(k + q)} \cdot u^{k + q - 1} \, du \\
    &= \lim_{q \to \frac{1}{2}-}\frac{1}{2q-1}\int_0^1\frac{u^{-2q + 1} - 1}{1 - u} \cdot \frac{(1 - \frac{u}{2})^{q - 1} - 1}{u} \, du \\
    &= \frac{1}{2} \int_0^1 \frac{\partial}{\partial q} \bigg[\frac{u^{-2q + 1} - 1}{1 - u} \cdot \frac{(1 - \frac{u}{2})^{q - 1} - 1}{u} \bigg]\bigg|_{q = 1/2} \, du \\
    &= \int_0^1 \frac{-\ln u}{1-u}\cdot\frac{\left(1-\frac{u}{2}\right)^{-1/2}-1}{u}\,du.
\end{align*} 
\end{remark}

\begin{proof}[Proof of Lemma~\ref{lem:T1-unif-conv}]
For any $m \ge 1$, we have
\begin{align}\label{eq:unif-conv-tail-bound} \nonumber
    \bigg|\sum_{k = m}^\infty &\frac{H(k, q) I(k, q)}{k^2} \cdot \int_0^1 \frac{u^{k+q-1} (1-u)^{1 - q}}{1 + u} \, du\bigg| \\ \nonumber
    &\le \sum_{k = m}^\infty \frac{|H(k, q) I(k, q)|}{k^2} \cdot \int_0^1 u^{k+\Re(q)-1} (1-u)^{1 - \Re(q)} \, du \\ \nonumber
    &= \sum_{k = m}^\infty \frac{|H(k, q) I(k, q)|}{k^2} \cdot \frac{\Gamma(k+\Re(q)) \Gamma(2-\Re(q))}{\Gamma(k + 2)} \\ \nonumber
    &= \sum_{k = m}^\infty \frac{|H(k,q)|}{k^2} \cdot \frac{\Gamma(k+1)}{|\Gamma(k+q)| |\Gamma(1-q)|} \cdot \frac{\Gamma(k+\Re(q)) \Gamma(2-\Re(q))}{\Gamma(k + 2)} \\
    &= \frac{\Gamma(2-\Re(q))}{|\Gamma(1-q)|} \sum_{k = m}^\infty \frac{|H(k,q)|}{k^2(k+1)} \cdot \frac{\Gamma(k+\Re(q))}{|\Gamma(k+q)|}.
\end{align}
Using the inequality $|\Gamma(z)| \le \Gamma(\Re(z))$, valid for $\Re(z) > 0$, we note that for any $q \in R_{\eta}$ and $k \ge 3$,
\[
    \Gamma(k + 2q)| \le \Gamma(k + 2 \Re(q)) \le \Gamma(k + 1),
\]
so that
\begin{equation}\label{eq:H-bdd}
    |H(k, q)| \le \frac{k}{|1 - 2q|} \bigg(\frac{|\Gamma(k + 2q)|}{\Gamma(k + 1)|\Gamma(2q)|} + 1\bigg) \le \frac{k}{|1 - 2q|} \bigg(\frac{1}{|\Gamma(2q)|} + 1\bigg) \le C_{\eta} k,
\end{equation}
where $C_\eta := \sup_{q \in R_{\eta}} \frac{1}{|1 - 2q|} \big(\frac{1}{|\Gamma(2q)|} + 1 \big) < \infty$.

On the other hand, using the well-known lower bound (see, e.g., \cite[Section~5.6]{olver2010nist})
\[
    |\Gamma(z)| \ge \frac{\Gamma(\Re(z))} {\sqrt{\cosh(\pi \Im(z))}},
\]
valid for $\Re(z) \ge \frac{1}{2}$, we have for any $q \in R_{\eta}$ and $k \ge 2$,
\begin{equation}\label{eq:gamma-ratio-bdd}
    \frac{\Gamma(k + \Re(q))}{|\Gamma(k + q)|} \le \sqrt{\cosh(\pi)}.
\end{equation}
Plugging in the bounds \eqref{eq:H-bdd} and \eqref{eq:gamma-ratio-bdd}, we see from \eqref{eq:unif-conv-tail-bound} that for all $m \ge 3$,
\[
    \bigg|\sum_{k = m}^\infty \frac{H(k, q) I(k, q)}{k^2} \cdot \int_0^1 \frac{u^{k+q-1} (1-u)^{1 - q}}{1 + u} \, du\bigg| \le C'_\eta \sum_{k = m}^\infty \frac{1}{k^2},
\]
for some constant $C'_{\eta} > 0$, which does not depend on $q$. This establishes the desired uniform convergence.
\end{proof}

\section*{Acknowledgements}
The research of SSM was partially supported by the Prime Minister Early Career Research Grant ANRF/ECRG/2024/006704/PMS from the Anusandhan National Research Foundation, Govt.~of India.

\bibliographystyle{alpha}
\bibliography{refs.bib}

\appendix
\section{Auxiliary results}\label{sec:aux}
\begin{lemma}\label{lem:phi_holomorphic}
Fix $\eta \in (0,\frac{1}{2})$. Define the function $\psi: R_\eta \to \bbC$, given by
\begin{equation}\label{eq:phi_defn}
    \psi(z):= \frac{1}{2z-1}\int_0^1\frac{u^{-2z + 1} - 1}{1 - u} \cdot \frac{(1 - \frac{u}{2})^{z - 1} - 1}{u} du.
\end{equation}
Then, $\phi$ is holomorphic on $R_\eta$.
\end{lemma}
\begin{remark}
    Note that the integral exists and is finite since the integrand has finite limits as $u \to 0$ and $1$.
\end{remark}
\begin{proof}
Define $\Psi : R_\eta \times [0,1] \to \bbC$ by
\begin{equation}\label{eq:Phi_defn}
    \Phi(z, u) := 
    \begin{cases} 
        \frac{1}{2z-1} \cdot \frac{u^{-2z + 1} - 1}{1 - u} \cdot \frac{(1 - \frac{u}{2})^{z - 1} - 1}{u} &\text{if}\,\, u \in (0,1),\\
        \frac{z-1}{2(2z-1)}  &\text{if}\,\, u=0,\\
        1 &\text{if}\,\, u=1.
    \end{cases}
\end{equation}
Note that $\Psi$ is continuous on $R_\eta \times [0,1]$. Moreover, for each fixed $u \in [0,1]$, $\Psi$ is holomorphic on $R_\eta$. Also, note that
\[
    \psi(z) = \int_0^1 \Psi(z,u) \, du.        
\]
Now a direct application of Theorem~5.4 of \cite[Chapter 2]{stein2003complex} proves the holomorphicity of $\psi$.
\end{proof}

\end{document}